\documentclass{amsart}
\usepackage{graphicx} 
\input cyracc.def

\usepackage{amssymb}
\usepackage{amsthm}

\usepackage{todonotes}

\usepackage{tikz-cd}
\usepackage{xcolor}

\theoremstyle{plain}
\newtheorem{thm}{Theorem}[section]
\newtheorem{lem}[thm]{Lemma}
\newtheorem{cor}[thm]{Corollary}
\newtheorem{prop}[thm]{Proposition}
\newtheorem{conj}[thm]{Conjecture}

\theoremstyle{definition}
\newtheorem{dfn}[thm]{Definition}
\newtheorem{Q}[thm]{Question}

\newtheorem{claim}[thm]{Claim}

\def\dnfo{\;\raise.2em\hbox{$\mathrel|\kern-.9em\lower.4em\hbox
{$\smile$}$}}

\def\dnf#1{\lower.9em\hbox{$\buildrel\dnfo\over{ \scriptstyle  #1}$}}

\def\dfo{\;\raise.2em\hbox{$\mathrel|\kern-.9em\lower.4em\hbox{$\smile$}
\kern-.72em\lower.07em\hbox{\char'57}$}\;}
\def\df#1{\lower1em\hbox{$\buildrel\dfo\over{\scriptstyle #1}$}}

\newcommand{\sA}{\mathcal{A}}
\newcommand{\sB}{\mathcal{B}}
\newcommand{\sC}{\mathcal{C}}

\newcommand{\sH}{\mathcal{H}}

\newcommand{\sM}{\mathcal{M}}

\newcommand{\sO}{\mathcal{O}}

\newcommand{\sT}{\mathcal{T}}

\newcommand{\converge}{\!\!\downarrow}

\newcommand{\nat}{\in\omega}

\newcommand{\seq}[1]{\langle #1 \rangle}

\newcommand{\+}[1]{\mathcal{#1}}
\renewcommand{\*}[1]{\mathbf{#1}}

\title{Failure Modes for Structural Highness Notions}
\author[Calvert]{Wesley Calvert}
\address[Calvert]{School of Mathematical and Statistical Sciences\\ Mail Code 4408\\ 1245 Lincoln Drive\\ Southern Illinois University\\ Carbondale, Illinois 62901\\USA}
\email{wcalvert@siu.edu}
\urladdr{http://lagrange.math.siu.edu/calvert}

\author[Franklin]{Johanna N.Y.\ Franklin}
\address[Franklin]{Department of Mathematics \\ Room 306, Roosevelt Hall \\ Hofstra University \\ Hempstead, NY 11549-0114 \\ USA}
\email{johanna.n.franklin@hofstra.edu}
\urladdr{http://www.johannafranklin.net}
\author[Turetsky]{Dan Turetsky}
\address[Turetsky]{School of Mathematics and Statistics \\ Victoria University of Wellington \\ Wellington \\ New Zealand}
\email{dan.turetsky@vuw.ac.nz}
\thanks{The work of the first and second authors was partially supported by NSF grant DMS-2404023.  The third author was supported by Marsden grant 20-VUW-104.}
\date{\today}

\begin{document}

\begin{abstract}
In \cite{CFT2023}, we defined several classes of degrees that are high in senses related to computable structure theory.  Each class of degrees is characterized by a structural feature (e.g., an isomorphism) that it can compute if such a feature exists.  In this paper, we examine each of these classes and characterize them based on what they do if no such object exists.

We describe, in particular, reticent, loquacious, and collegiate senses of being high.  These, respectively, reflect the case where a computation from the degree can give output only if the desired feature exists, the case where it will give output of some kind whether or not the feature exists, and the case where the degree will either compute the feature or the best available approximation to it.
\end{abstract}

\maketitle

\section{Introduction}

If we tell you a joke, our intended outcome is that you laugh.  If we are successful, that is exactly what will happen.  On the other hand, there are many possible outcomes if we are unsuccessful in telling a joke.  You could groan, have no reaction, or even become angry.\footnote{This analogy was suggested to the first author by communications scholar Jason Jarvis and is explored further in \cite{Calvert2011}.}

In \cite{CFT2023}, we described degrees that are high in certain structural senses, for instance, high for isomorphism.  A degree $\mathbf{d}$ is high for isomorphism if for any computable structures $\sA$ and $\sB$ that are actually isomorphic, $\mathbf{d}$ computes an isomorphism between them.  This is success.  In this paper, we explore the ways in which such a degree $\mathbf{d}$ can be unsuccessful --- that is, what happens if $\sA$ and $\sB$ are \emph{not} isomorphic.

In giving talks about \cite{CFT2023}, the authors have often been asked about these failure modes of degrees which are high for isomorphism.  Several particular modes have been proposed in these questions.  We focus on three of them. Degrees that are \emph{reticently} high for isomorphism, previously mentioned in \cite{CFT2023}, are those which will produce an isomorphism if the structures in question are isomorphic but nothing at all if they are not. Degrees that are \emph{loquaciously} high for isomorphism will always produce a purported isomorphism whether or not the structures in question are isomorphic. Finally, degrees that are \emph{collegiate} at a level $\alpha$ will produce functions which are plausibly approximate isomorphisms at this level.

\subsection{Background and definitions}

Our notation is standard; we refer the reader to \cite{ashknight} for general background on computable structure theory. The following definitions were given in \cite{CFT2023}.

\begin{dfn}
We call a degree $\mathbf{d}$ {\em high for paths} if for every nonempty $\Pi^0_1$ class $\+ P$ on Baire space,\footnote{If we restricted the $\Pi^0_1$ classes to Cantor space, these would be better known as the PA degrees, which are well studied.} the degree $\mathbf{d}$ computes an element of $\+ P$. Similarly, a degree $\mathbf{d}$ is called {\em uniformly high for paths} if there is a $D \in \mathbf{d}$ and a total computable $f$ such that for every nonempty $\Pi^0_1$ class of functions $\+ P_i$, the function $\{f(i)\}^D$ is an element of $\+ P_i$.
\end{dfn}

Note that instead of phrasing this in terms of computing elements of nonempty $\Pi^0_1$ classes, we could phrase this as computing paths through ill-founded trees on $\omega^{<\omega}$.  We will use both of these interpretations throughout the paper.

\begin{dfn}
We call a degree $\mathbf{d}$ {\em high for isomorphism} if for any two computable structures $\+ M$ and $\+ N$ with $\+ M \cong \+ N$, there is a $\mathbf{d}$-computable isomorphism from $\+ M$ to $\+ N$. Similarly, we call a degree $\mathbf{d}$ {\em uniformly high for isomorphism} if there is a $D \in \mathbf{d}$ and a total computable $f$ such that for any computable structures $\+M_i \cong \+M_j$, the function $\{f(i,j)\}^D$ is an isomorphism from $\+M_i$ to $\+M_j$.
\end{dfn}

We made limited use in \cite{CFT2023} of a local version of these notions, which we will use more freely here.

\begin{dfn}
    Given a class $K$ of countable structures, we call a degree $\mathbf{d}$ \emph{high for isomorphism for $K$} if for any two $\+ M, \+ N \in K$ with $\+ M \cong \+ N$, there is a $\mathbf{d}$-computable isomorphism from $\+ M$ to $\+ N$.  We call a degree $\mathbf{d}$ \emph{uniformly high for isomorphism for $K$} under the natural uniformization of these conditions.
\end{dfn}

Now we turn our attention to the ability to compute infinite descending sequences within ill-founded linear orders.

\begin{dfn}
If $\+L$ is a linear order, a {\em tight descending sequence} in $L$ is an infinite descending sequence which is unbounded below in the ill-founded part of $\+L$, i.e., if $f$ is a tight descending sequence and $g$ is any descending sequence, then for every $n$ there is an $m$ with $f(m) <_{\+L} g(n)$.
\end{dfn}

\begin{dfn}
A degree $\mathbf{d}$ is {\em high for (tight) descending sequences} if any computable ill-founded linear order $\+L$ has a $\mathbf{d}$-computable (tight) descending sequence.
\end{dfn}

Degrees that are uniformly high for (tight) descending sequences are defined analogously.

\subsection{Plan of the Paper} For various notions of structural highness, we consider three variants.  In Section \ref{SecReticent}, we describe degrees which are high in a \emph{reticent} sense, including a review of some material from \cite{CFT2023}, but adding a new result on when reticent highness for isomorphism for a class implies reticent highness for isomorphism in general.

In Section \ref{SecLoquacious}, we consider a new variant, the degrees which are \emph{loquaciously} high --- that is, the highness is witnessed by a total functional in a way we will make precise.  This section leaves open two problems that we consider very interesting.

Section \ref{SecCollegiate} describes a form of graceful degradation in the form of degrees which are \emph{collegiately} high.  These degrees, if unable to compute what we had hoped they might, will compute some approximation to it.  The degrees are graded by the quality of the approximation.  We find that, globally, the hierarchy of notions collapses, but that at least on Abelian $p$-groups the hierarchy is proper.

This leaves the reticent and loquacious variants for global nontrivial comparison.  In Section \ref{SecRelationships}, we describe implications (or lack thereof) between these notions.

In the brief Section \ref{SecNonlow}, we show that all notions presented in the paper collapse for nonlowness for isomorphism.

Finally, Section \ref{sec:other.stuff} collects additional results on structural highness notions that many colleagues have asked about but that do not seem likely to be published elsewhere in the short term.

\section{Degrees that will never speak unless they have something to say}\label{SecReticent}

We begin by recalling the definition of reticence from \cite{CFT2023}. This definition is written here in full generality so that we can later apply it to highness in different contexts.

\begin{dfn}
Let $\*d$ be a Turing degree which is uniformly high for a given notion $\+N$; that is, there is some $I \subseteq \omega$ (depending on $\+N$) and a $\*d$-computable function $F:\omega\rightarrow\omega^\omega$ such that if $i \in I$, then $F(i)$ has some $\+N$-appropriate relation to $i$.

We say that $\*d$ is {\em reticently high for $\+N$} if $F$ can be chosen such that for every $i \not \in I$, the partial function $F(i)$ diverges for all inputs.
\end{dfn}

This notion is the same as the notion we called \emph{uniformly high for $\+N$ in the reticent sense} in \cite{CFT2023}. We choose a shorter term here because there is not the same need to emphasize uniformity in the context of this paper. 

Now we give an example. If we are considering uniform highness for paths, we will have an effective enumeration $(\mathcal{P}_i : i\in\omega)$ of $\Pi^0_1$ classes. Our $I$ will be the set of indices for nonempty $\Pi^0_1$ classes, and $F(i)$ will be an element of the $i^{th}$ $\Pi^0_1$ class if such an element exists. If no such element exists, then $F$ diverges at $i$. Similarly, if we are considering uniform highness for isomorphism, we will have an effective enumeration $( (\sA_i, \sB_i) : i\in\omega)$ of pairs of computable structures, and our $I$ will be the set of $i$ such that $\sA_i$ and $\sB_i$ are isomorphic. Our $F(i)$ will then be an isomorphism from $\sA_i$ to $\sB_i$ if one exists; if $\sA_i$ and $\sB_i$ are not isomorphic, then $F$ diverges at $i$.

In \cite{CFT2023}, we found the following characterizations of reticent highness for various classes:

\begin{thm}[\cite{CFT2023}, Theorem 5.2]\label{thm:reticence_all_the_same}
The following properties of a degree $\mathbf{d}$ are equivalent:
\begin{enumerate}
    \item $\mathbf{d}$ enumerates all $\Sigma^1_1$ sets.
    \item $\mathbf{d}$ is reticently high for paths.
    \item $\mathbf{d}$ is reticently high for isomorphism.
    \item $\mathbf{d}$ is reticently high for descending sequences.
    \item $\mathbf{d}$ is reticently high for tight descending sequences.
    \item $\mathbf{d}$ is uniformly high for tight descending sequences.
\end{enumerate}
\end{thm}

This is notable because very few of these notions coincide in the nonreticent sense, uniformly or not. For instance, highness for paths and highness for isomorphism coincide \cite{CFT2023}, but neither highness for isomorphism nor highness for tight descending sequences implies the other (see Section \ref{sec:other.stuff}).

We now characterize reticently high for isomorphism in terms of reticence for a particular type of class of structures.

\begin{thm}\label{RHFIComplete} If $K$ is a class of structures with a $\Sigma^1_1$-complete isomorphism problem and $\mathbf{d}$ is reticently high for isomorphism for $K$, then $\mathbf{d}$ is reticently high for isomorphism.
\end{thm}

\begin{proof}  Suppose that $\mathbf{d}$ is reticently high for isomorphism for $K$ and that $K$ is a class of structures with a $\Sigma^1_1$-complete isomorphism problem.  By Theorem \ref{thm:reticence_all_the_same}, it suffices to show that $\mathbf{d}$ enumerates all $\Sigma^1_1$ sets.  Let $X \in \Sigma^1_1$.  Now we have a uniformly computable sequence $\left((\sA_i,\sB_i):i \nat\right)$ of pairs of structures from $K$ such that $\sA_i \cong \sB_i$ if and only if $i \in X$.

Let $F:\omega \to \omega^\omega$ be a $\mathbf{d}$-computable function witnessing that $\mathbf{d}$ is reticently high for isomorphism; that is, if $i$ is the index for a pair of isomorphic computable structures, then $F(i)$ is an isomorphism, and $F$ diverges otherwise.

Then we can enumerate $X$ as follows:  When we see $F(i)$ converge, we enumerate $i \in X$.
\end{proof}

\section{Degrees that always have something to say}\label{SecLoquacious}

\begin{dfn}  Let $\*d$ be a Turing degree that is uniformly high for a given notion $\+N$ as in the previous section.  We say that $\*d$ is {\em loquaciously high for $\+N$} if $F$ can be chosen such that $F(i)$ is total for all $i \in \omega$.\end{dfn}

For instance, we say that a degree $\mathbf{d}$ is loquaciously high for paths if there is a total $\*d$-computable function $F$ such that if the computable tree $T_n$ is ill founded, then $F(n)$ gives a path, and if $T_n$ is well founded, then $F(n)$ gives some element of Baire space (not necessarily with any relation to $T_n$).

Similarly, $\*d$ is loquaciously high for isomorphism if there is a total $\*d$-computable function $F$ such that if $\+M_i, \+M_j$ are isomorphic computable structures, then $F(i, j)$ gives an isomorphism between them, and if $\+M_i, \+M_j$ are not isomorphic, then $F(i, j)$ is at least total.

Note that every PA degree is loquaciously PA.  Indeed, we could computably prepare a computably given tree on $2^{<\omega}$ so that the original tree is unchanged if infinite and changed to an infinite tree otherwise.

\begin{prop}\label{pathsvsisom}
    A degree $\mathbf{d}$ is loquaciously high for paths if and only if it is loqaciously high for isomorphism.
\end{prop}

\begin{proof}
This proof follows that of Proposition 2.6 in \cite{CFT2023}. We include it here for completeness and make careful note of the aspects of the proof that allow us to maintain loquacity.

Suppose that $\* d$ is loquaciously high for paths.  Given two isomorphic computable structures $\+M_i$ and $\+M_j$,  we define a tree $I_{i,j}$ as follows.  A vertex of the tree consists of a pair $(\sigma,\tau)$ such that $\sigma, \tau$ are partial isomorphisms $\sigma:\+M_i \to \+M_j$ and $\tau:\+M_j \to \+M_i$ whose domains are finite initial segments of the universes of the respective structures (and of the same length), and $\sigma$ and $\tau$ are inverses where both are defined. Vertices are ordered by extension.  Note that a path through $I_{i,j}$ represents a pair consisting of an isomorphism and its inverse and that this tree is uniformly computable in the two structures.  Now $\* d$, via a total functional, computes a path through $I_{i,j}$, which gives an isomorphism and its inverse.

Let $\*d$ be loquaciously high for isomorphism, and let $\sT$ be a nonempty $\Pi^0_1$ class of functions generated by a computable tree $T$.

Our proof is based on rank-saturated trees of infinite rank \cite{CalvertKnightMillar2006,FFHKMM}; we will require the following facts:
\begin{enumerate}
    \item There is a computable rank-saturated tree $S$ of infinite rank that has a computable path.  (This follows from the proof of Lemma 1 in \cite{FFHKMM}, beginning with a Harrison order with a computable descending sequence.) 
    \item Any two computable rank-saturated trees of infinite rank are isomorphic.  (Proposition 2, \cite{FFHKMM})
    \item If $S$ is a computable rank-saturated tree of infinite rank and $T$ is any computable tree with a path, then $T\star S$ is a computable rank-saturated tree of infinite rank.  (Proposition 1, \cite{FFHKMM})
\end{enumerate}

Then $S$ and $T\star S$ are our two computable structures in some appropriate language.  Let $f$ be a computable path through $S$.  Since $\*d$ can, via a total functional, compute an isomorphism between $S$ and $T\star S$, we have that $\*d$ can carry $f$ over to a path in $T\star S$ and then project that down to a path in $T$ as required.  

Note that if $T$ is well founded, then $T\star S$ will be as well, so there will be no isomorphism between $S$ and $T\star S$.  In this case, our total $\*d$-computable function will give us a total function that is not an isomorphism, which we apply to the computable path through $S$, obtaining a function which is not a path of $T\star S$.
\end{proof}

\begin{cor}\label{cor:loquacious_inverses}  Suppose $\mathbf{d}$ is loquaciously high for isomorphism.
\begin{enumerate}
    \item There are $\*d$-computable functions $\Phi$ and $\Psi$ such that if $\sM_i \cong \sM_j$ are computable structures, then $\Phi(i,j): \sM_i \cong \sM_j$ and $\Psi(i,j) = \Phi(i, j)^{-1}$.
    \item There is a $\*d$-computable function $\Theta$ witnessing loquaciously high for isomorphism and such that for any computable $\sM_i \cong \sM_j$, this same $\Theta$ gives the inverse as $\Theta^{-1}(j,i)=\Theta(i,j)$.
\end{enumerate}
\end{cor}

We note that this statement is not as trivial as it might first appear.  As $\*d$ is loquaciously high for isomorphism, we have a total $\*d$-computable function $\Phi$ giving isomorphisms, and the naive approach is to define $\Psi$ by simply computing the inverse of the isomorphism given by $\Phi$.  The difficulty is that when $\sM_i \not \cong \sM_j$, $\Phi(i, j)$ is producing garbage which may in particular fail to be a bijection.  In that case, attempting to compute an inverse may result in something partial, which we cannot have.

\begin{proof} Suppose $\mathbf{d}$ is loquaciously high for isomorphism.  It follows that $\mathbf{d}$ is also loquaciously high for paths.  Let $\sM_i,\sM_j$ be computable structures.  We construct the tree $I_{i,j}$ as in the previous proof.

From $\mathbf{d}$, we can, uniformly in $i$ and $j$, compute a path in $T_{i,j}$, and the two components of this path give the necessary isomorphisms.  This establishes the first part.

    Toward the second part, we let $\Phi$ and $\Psi$ be as in Part 1.  Now we define $\Theta(i,j) = \Phi(i,j)$ if $i < j$, $\Theta(i,j) = \Psi(i,j)$ if $i > j$, and set $\Theta(i,j)$ to be the identity if $i = j$.
\end{proof}

Indeed, the degrees loquaciously high for paths are quite powerful.  We recall that there are degrees which are high for paths without further restriction which have $\sO$ for their third jump.  We have seen in Theorem \ref{thm:reticence_all_the_same} that the degrees reticently high for paths are exactly the ones which enumerate all $\Sigma^1_1$ sets.  The status in this regard of degrees uniformly high for paths is one of the principal open problems in this area.

\begin{prop}\label{prop:lhfp_enumerates_O} If $\mathbf{d}$ is loquaciously high for paths, then $\mathbf{d}$ enumerates all $\Pi^1_1$ sets uniformly.\end{prop}

\begin{proof}  Suppose that $\mathbf{d}$ is loquaciously high for paths.  It suffices that $\mathbf{d}$ can enumerate the indices of well-founded computable trees.  

Let $\Phi$ witness that $\mathbf{d}$ is loquaciously high for paths.  Now, when $i$ is the index of a computable tree $T$, we search for a length $\ell$ at which the path $\Phi(i)\!\!\upharpoonright_\ell\ \notin T$.  

If such an $\ell$ arises, then $T$ is well founded, since $\Phi(i)$ would have been a path of $T$ if $T$ were ill founded.  In that case, $T$ should be enumerated as a well-founded computable tree.\end{proof}

\begin{prop} There is a degree that enumerates $\sO$ but is not loquaciously high for paths.\end{prop}

\begin{proof}

We carry out a forcing in which a condition is a finite $\sigma: \omega \times \omega \to \{0,1\}$ such that if there is an $s$ such that $\sigma(n,s) = 1$, then $n \in \mathcal{O}$.  Note that being a condition for this forcing is $\Pi^1_1$.  A sufficiently generic filter $\+F$ will give us $X = \bigcup\limits_{\sigma \in \+F} \sigma$ such that $n \in \+O$ if and only if there is an $s$ such that $X(n, s) = 1$, so $X$ will enumerate $\+O$.

Suppose we have a condition $\sigma$ and a functional $\Phi$.  We wish to show that we can extend $\sigma$ to a condition that forces that $X$ is not loquaciously high for paths via $\Phi$.  If there is no extension $\tau \preceq \sigma$ with $\Phi^\tau(e; 0)\converge$, then $\sigma$ has already forced that $\Phi$ is partial.  Otherwise, we will build a tree $T_e$ by application of the Recursion Theorem.  Let $A_e = \{ m : (\exists \tau \le \sigma) [\Phi^\tau(e;0) = m]\}$.  The set $A_e$ is nonempty and $\Pi^1_1$.  By $\Pi^1_1$-selection, we can uniformly obtain a $\Pi^1_1$ singleton $B_e \subseteq A_e$.  For each $m$, let $S_{e,m}$ be a computable tree which is well founded if and only if $m \in B_e$.  Consider the tree made by putting a common root below all the $S_{e,m}$ trees so that $S_{e,m}$ is the tree below the $m$th child.  By the Recursion Theorem, there is some $e$ such that this tree is the tree with index $e$, which we denote by $T_e$.

Let $m$ be the unique member of $B_e$.  Then the $m$th child of the root is the only non-extendable child, but there is some extension $\tau$ with $\Phi^\tau(e)$ going through $m$.  So $\tau$ forces that $\Phi$ is not a witness for loquaciously high for paths.

Thus, if $X$ is generated from a sufficiently generic filter $\+F$ as described above, $X$ will not be loquaciously high for paths via any Turing functional $\Phi$, and thus will not be loquaciously high for paths.  As mentioned above, such an $X$ will enumerate $\+O$, so it gives us our desired degree.
\end{proof}

At this point, it may well be asked whether the degrees loquaciously high for paths (equivalently, isomorphism) are \emph{too} powerful to be interesting.  We might be inclined to suspect that they all compute $\sO$.  As we will see, this is not the case.

The following Lemma is closely related to Lemma 4.46 in \cite{CFT2023}, and we need only add that the degree constructed is \emph{loquaciously} high, rather than merely high.

\begin{lem}\label{FGFisLHDS}
    Let $(\mathcal{L}_n : n \in \omega)$ be an standard listing of the computable linear orders.  Let $f$ be a function such that if $\mathcal{L}_n$ is ill founded, then $f^{[n]}$ majorizes a descending sequence of $\+L_n$.  Then $f$ is loquaciously high for descending sequences.
\end{lem}

\begin{proof}
    Given $n$, we use $f^{[n]}$ to attempt to construct a descending sequence.  Define the sequence $(a_i : i \in \omega)$ as follows:
    \begin{itemize}
        \item Consider the first $f^{[n]}(0)$ elements of $\+L_n$ to be enumerated, and let $a_0$ be the rightmost of these.
        \item Given $a_i$, consider the first $f^{[n]}(i+1)$ elements of $\+L_n$ to be enumerated, and let $a_{i+1}$ be the rightmost of these which is strictly to the left of $a_i$.
    \end{itemize}
    Observe that this is uniform.  During this construction, we may reach a point where we are unable to continue because there is no element left of $a_i$ within the first $f^{[n]}(i+1)$ elements.  Crucially, we can recognize when this happens.  In this case, we then know that $\+L_n$ is well founded, so we simply switch to outputting 0, and loquaciousness is maintained.

    Now suppose $\+L_n$ is ill founded and $(b_n : n \in \omega)$ is the descending sequence majorized by $f^{[n]}$.  Then, by induction, we have $b_i \le_{\+L_n} a_i$.  Thus, $b_{i+1}$ will always be an element among the first $f^{[n]}(i+1)$ elements of $\+L_n$ and strictly to the left of $a_i$, so the construction can continue.
\end{proof}

\begin{prop}
    Let $(U_n : n \in \omega)$ be an effective listing of $\Sigma^1_1$ subsets of $\omega$.  A degree $\*d$ is loquaciously high for descending sequences if and only if there is a total $\*d$-computable function $f$ such that if $U_n$ is nonempty, then $f(n) \ge \min U_n$.
\end{prop}

\begin{proof}
    We proceed as in the proof of Lemma 4.14 in \cite{CFT2023}, noting that loquaciousness makes the function total.
\end{proof}

A similar adaptation of the proof of Lemma 4.15 in \cite{CFT2023}, in combination with our observation before Lemma \ref{pathsvsisom} that every PA degree is loquaciously PA, gives us the following:

\begin{lem}\label{PAoverLHDS} Let $\mathbf{d}$ be PA over a degree that is loquaciously high for descending sequences.  Then $\mathbf{d}$ is loquaciously high for isomorphism.
\end{lem}

We can now combine these results to get the following separation.

\begin{thm}\label{LHFInotO} There is a degree $\mathbf{d}$ that is loquaciously high for isomorphism such that $\mathbf{d}'\equiv_T \sO$.\end{thm}

\begin{proof}
    By Lemmas \ref{FGFisLHDS} and \ref{PAoverLHDS} (and the Low Basis Theorem), it suffices to make a sufficiently fast-growing function which jumps to $\sO$.

    Let $f$ be such that if $\+L_n$ is ill founded, then $f^{[n]}$ majorizes a descending sequence.  Note that $\sO$ can compute such an $f$.  We build a $g$ majorizing $f$ using the partial order of Hechler forcing.  We start with the condition $(\seq{}, f)$.

    Given a condition $(\sigma, h)$ and an $e \in \omega$, we ask if there exists a function $\hat{h}$ extending $\sigma$ such that there is no $\tau$ extending $\sigma$ and majorizing $\hat{h}$ with $\Phi_e^\tau(e)\converge$ (crucially, this question makes no reference to $h$).  Note that this is $\Sigma^1_1$, and so $\sO$ can answer it.  If there is such an $\hat{h}$, then the collection of all such is a $\Sigma^1_1$ class, and so $\sO$ can compute an element of it.  We then make our next condition $(\sigma, \max(h, \hat{h}))$.

    If there is no such $\hat{h}$, then in particular there is a $\tau$ extending $\sigma$ and majorizing $h$ with $\Phi_e^\tau(e)\converge$.  In this case, we simply search for such a $\tau$ and define our next condition to be $(\tau, h)$.

    We perform the above for every $e$, interspersed with steps where we extend the length of $\sigma$.  This is all $\sO$-computable, so we obtain an $\sO$-computable sequence of conditions $(\sigma_0, h_0) \succeq (\sigma_1, h_1) \succeq \dots$.  Let $g = \bigcup_i \sigma_i$.  Then $g$ majorizes $f$, and we have been careful to force the jump, so $g' \le_T \sO$.
\end{proof}

\begin{prop} Let $\mathbf{d}$ be loquaciously high for descending sequences.  Then $\mathbf{d}$ enumerates all $\Pi^1_1$ sets.
\end{prop}

\begin{proof} It is sufficient to enumerate all well-founded linear orderings.  We give $\mathbf{d}$ the index for a linear ordering and examine the output.  If the sequence it produces ever fails to be decreasing, then we enumerate that index as the index of a well-founded linear ordering.
\end{proof}

The following result was implicit in \cite{CFT2023}, although we did not have the terminology at hand at the time.

\begin{prop} There is a degree that is loquaciously high for descending sequences but not high for paths.\end{prop}

\begin{proof} In the proof of Corollary 4.47 in \cite{CFT2023}, we constructed a degree that is uniformly high for descending sequences but not high for paths.  The construction worked by producing a total function that majorized a descending sequence in every ill-founded linear order.  By Lemma~\ref{FGFisLHDS}, such a function is loquaciously high for descending sequences.\end{proof}

\begin{dfn}
    Let $K$ be a class of structures.  We define \emph{loquaciously high for isomorphism for $K$} in the obvious way, except that in light of Part 2 of Corollary~\ref{cor:loquacious_inverses}, we require the witnessing $\Theta$ to have the property that for $\sM_i \cong \sM_j \in K$, then we have $\Theta(i,j) = \Theta^{-1}(j,i)$.
\end{dfn}

\begin{lem}
    Let $K$ be a class of structures.  If a degree $\mathbf{d}$ is uniformly high for isomorphism for $K$ and enumerates $\sO$, then $\mathbf{d}$ is loquaciously high for isomorphism for $K$.
\end{lem}
\begin{proof}
    By enumerating $\sO$, we can enumerate those pairs of non-isomorphic structures.  Thus, given two structures $\sM_i, \sM_j \in K$, we can attempt to compute an isomorphism via the uniform process, and if we eventually see that the structures are non-isomorphic, we can switch to outputting garbage.  In this way we define $\Theta(i,j)$ for $i < j$.  We define $\Theta(i,j)$ for $j > i$ by using unbounded search to find the inverse of $\Theta(j,i)$, switching to outputting garbage when we see $\sM_i \not \cong \sM_j$.  Of course, we define $\Theta(i,j)$ to be the identity when $i = j$.
\end{proof}

Of course, being loquaciously high for isomorphism for $K$ implies (by definition) being uniformly high for isomorphism for $K$, but in cases where the isomorphism problem for $K$ is not too difficult, such a degree may not enumerate $\sO$.  However, the converse holds with an additional hypothesis.

\begin{lem}
    Let $K$ be a class of structures for which the isomorphism problem is $\Sigma^1_1$-complete. In that case, if a degree is loquaciously high for isomorphism for $K$, then it is uniformly high for isomorphism for $K$ and enumerates $\sO$.
\end{lem}
\begin{proof}
    That the degree is uniformly high for isomorphism for $K$ is by definition.  To enumerate $\sO$, we use the completeness of the isomorphism problem.  Given $x \in \omega$, we uniformly construct two computable structures $\sA_x^0, \sA_x^1 \in K$ such that $\sA_x^0 \cong \sA_x^1$ if and only if $x \not \in \sO$.  Now, from our degree, we compute total functions $f: \sA_x^0 \to \sA_x^1$ and $g: \sA_x^1 \to \sA_x^0$ which are inverse isomorphisms if and only if $x \not \in \sO$.  Checking if two functions are inverses and structure preserving is $\Pi^0_1$.
\end{proof}

Of course, if we let $K$ be the class of all computable structures, the same result holds for degrees which are, without restriction to a class, loquaciously high for isomorphism via the same proof.  The proof only requires minor modification to work for loquaciously high for paths.

We now mention two related open problems in this area.

\begin{conj}
    If every degree loquaciously high for isomorphism for $K$ is uniformly high for isomorphism, then the isomorphism problem for $K$ is $\Sigma^1_1$-complete.
\end{conj}

\begin{Q}\label{ELHFI}
    Suppose that $K$ is a class of structures for which the isomorphism problem is $\Sigma^1_1$-complete and that $\mathbf{d}$ is loquaciously high for isomorphism for $K$.  Then is $\mathbf{d}$ loquaciously high for isomorphism?
\end{Q}

Note that Question \ref{ELHFI} can be answered in the affirmative for reticence, as we saw in Theorem \ref{RHFIComplete}.

Now we recall from \cite{HTMMM2017} the definition of a computable functor and that of effective natural isomorphism.  These two notions give a sense of encoding one class of structures in another that is sufficient to transfer many important properties in computable structure theory.

\begin{dfn} Let $K,K'$ be categories and $F:K \to K'$ a functor.  We say that $F$ is \emph{computable} if there exist two computable Turing functionals $\Phi,\Phi_*$ such that 
\begin{itemize}
    \item For every $\sB \in K$, the function computed by $\Phi^{D(\sB)}$ gives the atomic diagram of $F(\sB)$, and
    \item For every morphism $h:\sB \to \sC$ of elements of $K$, the function computed by $\Phi_*^{D(\sB) \oplus  h\oplus D(\sC)}$ is the function $F(h)$.
\end{itemize}\end{dfn}

We now show that reductions via appropriate computable functors also transfer the properties discussed in the present paper.  We note that the hypotheses of this theorem are weaker than the universal transformal reduction described in \cite{HTMMM2017}, so the theorem applies to any of the cases in which that relation is known.

\begin{thm}\label{functorial}
 Let $K$ and $\widetilde{K}$ be classes of structures with the classical morphisms between those structures.  Suppose we have a subclass $J \subseteq K$ and computable functors $F:K \to J$ and $G:J \to K$ such that there is a Turing functional $\Lambda$ such that for every $\sA \in K$, the functional $\Lambda$ computes an isomorphism $\Lambda^{\sA}: G\circ F(\sA) \to \sA$.  Then every degree that is (uniformly, loquaciously, reticently) high for isomorphism for $\widetilde{K}$ is (uniformly, loquaciously, reticently) high for isomorphism for $K$.
\end{thm}

\begin{proof} Let $\mathbf{d}$ be high for isomorphism for elements of $\widetilde{K}$, and let $\Lambda$ witness the effective natural isomorphism to the identity.  We will show that $\mathbf{d}$ is high for isomorphism for elements of $K$.  

Let $\sA_1,\sA_2$ be isomorphic computable objects of $K$ with $\sB_i =F(\sA_i)$.  Now $\mathbf{d}$ computes an isomorphism $f:\sB_1 \to \sB_2$.  Since $G$ is computable, $G(f)$ is computable from $f$.  Since $G$ is a functor, $G(f):G(\sB_1) \to G(\sB_2)$ is an isomorphism, and $\mathbf{d}$ computes it.   Then $\Lambda^{\sA_2} \circ G(f) \circ \left(\Lambda^{\sA_1}\right)^{-1}: \sA_1 \to \sA_2$ is an isomorphism and still computable from $\mathbf{d}$.

The additions of uniform and reticent features are clear (e.g.\ if $\mathbf{d}$ computes the required things uniformly, the rest of the process is also uniform).  Loquaciousness is more subtle, since $G(f)$ need not be total if $f$ is garbage.  Given $f$ and $g$ which are possibly inverse isomorphisms, we compute $G(f)$ (and thus $\Lambda^{\sA_2} \circ G(f) \circ \left(\Lambda^{\sA_1}\right)^{-1}$) while simultaneously checking if $f$ and $g$ are inverses and $f$ is an embedding; if we see either fail, we switch to outputting garbage.
\end{proof}

\begin{cor} The following classes each have the property that every degree that is (uniformly, loquaciously, reticently) high for isomorphism for that class is (uniformly, loquaciously, reticently) high for isomorphism:
\begin{enumerate}
    \item symmetric, irreflexive graphs,
    \item partial orderings,
    \item lattices,
    \item rings (with zero-divisors),
    \item integral domains of arbitrary characteristic,
    \item commutative semigroups, and
    \item 2-step nilpotent groups.
\end{enumerate}\end{cor}

\begin{proof}
    It is noted in \cite{HTMMM2017} that the reductions described in \cite{hkss02} are uniform transformal reductions, which, as we have said, is stronger than the hypotheses for Theorem \ref{functorial}.
\end{proof}

\section{Degrees that will always give it their best}\label{SecCollegiate}

Perhaps the most attractive failure mode to hope for is that we might be able to compute, if not an isomorphism, at least the closest thing available.  We might expect that computing, for instance, $\Sigma_1$-elementary embeddings would be relatively easy.  We will see that even computing embeddings with a very slight degree of elementarity is as difficult as computing isomorphisms.  We even find that every degree that is high for isomorphism can compute $\Sigma_\alpha$-elementary embeddings, if they exist, for any $\alpha$.  This result will be Proposition \ref{CollegiateCollapse}.

\begin{dfn}
\begin{enumerate}
    \item     A degree $\mathbf{d}$ is \emph{$\alpha$-collegiate} if, for any ordinal $\beta \le \alpha$ and any computable structures $\+A$ and $\+B$, $\*d$ computes a $\Sigma_\beta$-elementary embedding of $\+A$ into $\+B$ if such an embedding exists. 

    \item A degree $\mathbf{d}$ is said to be \emph{computably collegiate} if it is $\alpha$-collegiate for every computable ordinal $\alpha$.
\end{enumerate}

\end{dfn}

The following result is expressed in the language of back-and-forth relations as defined, for instance, in Section 15.1 of \cite{ak00}.

\begin{prop}\label{prop:omega1ck.implies.isomorphic}
    For computable structures, $\sA \equiv_{\omega_1^{ck}} \sB$ implies $\sA \cong \sB$.
\end{prop}
\begin{proof}
    The $\exists$ player has a winning strategy for the Ehrenfeucht-Fra\"iss\'e game of length $\alpha$ for every computable $\alpha$.  By $\Sigma^1_1$-bounding, they must have a winning strategy for a game of some ill-founded length.  By having the $\forall$ player play along a descending sequence, we can extract an isomorphism.
\end{proof}

\begin{prop}\label{CollegiateHFP} Let $\alpha$ be a computable ordinal.  Then every $\alpha$-collegiate degree is high for paths.\end{prop}

\begin{proof}
This follows the proof of Proposition 2.6 in~\cite{CFT2023}, where we showed that a high for isomorphism degree is high for paths.  In that proof, we constructed isomorphic trees and used the fact that an isomorphism between the trees carries a path through one tree to a path through the other tree.  Since the trees are isomorphic, there is a $\Sigma_0$-elementary embedding (that is, an embedding) from one to the other, and an embedding will also carry a path through one tree to a path through the other.
\end{proof}

We observe that, as a consequence of Proposition \ref{CollegiateHFP}, every $\alpha$-collegiate degree is also high for descending sequences.

\begin{prop} Every computably collegiate degree is $\alpha$-collegiate for every ordinal $\alpha$.
\end{prop}

\begin{proof} Let $\mathbf{d}$ be computably collegiate, $\beta$ be an ordinal, and $\sA, \sB$ be computable structures such that there is a $\Sigma_\beta$-elementary embedding of $\sA$ into $\sB$.  Thus $\sA \equiv_\beta \sB$.  We must show that $\mathbf{d}$ computes a $\Sigma_\beta$-elementary embedding of $\sA$ into $\sB$.

Of course, if $\beta < \omega_1^{CK}$ there is nothing left to do, since a computably collegiate degree has this property by definition.  For $\omega_1^{CK}\le \beta$, we have that $\sA \cong \sB$ by Proposition~\ref{prop:omega1ck.implies.isomorphic}.  By Proposition~\ref{CollegiateHFP}, $\*d$ computes an isomorphism from $\sA$ to $\sB$, which is a $\Sigma_\beta$-elementary embedding.
\end{proof}

Note that Proposition~\ref{CollegiateHFP} is subsumed by the next result.

\begin{prop}\label{CollegiateCollapse}
    For any computable $\alpha$, the following properties of a degree $\mathbf{d}$ are equivalent:
    \begin{enumerate}
        \item $\mathbf{d}$ is $\alpha$-collegiate,
        \item $\mathbf{d}$ is computably collegiate, and
        \item $\mathbf{d}$ is high for isomorphism.
    \end{enumerate}
    The same is true for the uniform versions of these notions.
\end{prop}

\begin{proof} That (1) and (2) each imply (3) is exactly Proposition \ref{CollegiateHFP}.  To show that a degree high for isomorphism is $\alpha$-collegiate, note that the $\Sigma_\beta$-elementary embeddings form a $\Sigma^1_1$ class, so that if $\mathbf{d}$ is high for isomorphism, it can compute an element of this class.  Moreover, this argument can be made for any computable $\alpha$, so we have that (3) implies (2).  All of these steps are uniform.
\end{proof}

While this result seems to eliminate the relevance of the collegiate degrees, they may still be hierarchical in restricted classes of structures.

\begin{prop}\label{collegiatehierarchy} There is a sequence of ordinals $(\alpha_i)$ cofinal in $\omega_1^{CK}$ such that for each $i<j$ there exists a degree $\mathbf{d}_{ij}$ which is $\alpha_i$-collegiate for reduced Abelian $p$-groups but not $\alpha_j$-collegiate for reduced Abelian $p$-groups.
\end{prop}

\begin{proof} Computing $\alpha_i$ embeddings can be done by an Ulm back-and-forth argument up to the right ordinal height \cite{BarkerAPG,CalvertAPG}.\end{proof}

Finally, we ask about an added level of uniformity, which comes perhaps closest to our original intent with the collegiate degrees --- degrees which are uniformly high for isomorphism, with their collegiate character as a failure mode.

\begin{dfn} A degree $\mathbf{d}$ is said to be \emph{hyperuniformly $\alpha$-collegiate} if it is uniformly high for isomorphism and the uniform procedure witnessing this fact has the further property that if there is a $\Sigma_\alpha$-elementary embedding from $\sA$ to $\sB$, then that procedure outputs such an embedding (note that if $\sA \cong \sB$, then the isomorphism the procedure outputs is such a $\Sigma_\alpha$-elementary embedding).\end{dfn}

Certainly we know of these degrees that they must all be loquaciously high for isomorphism, and we know that any degree computing $\sO$ is hyperuniformly $\alpha$-collegiate for any computable $\alpha$.

\begin{prop} For any computable $\alpha$, the hyperuniformly $\alpha$-collegiate degrees are exactly the ones that can compute $\+O$.\end{prop}

\begin{proof}
    We show that such a degree can distinguish computable well-orders from Harrison orders and thus compute $\+O$.

    First, let $T_\infty$ be the computable rank saturated tree with a path. Let $S_\alpha$ be a tree of Scott rank $\omega_1^{ck}$, which is rank-saturated for ranks $\le \omega\alpha$, and let $T_\beta$ be a rank saturated tree of rank $\beta$, and similarly $T_{\beta+1}$.

    We build computable structures $\+A$ and $\+B$.  $\+A$ will consist of infinitely many copies of $T_{\beta+1}$ and infinitely many copies of $S_\alpha$, and  $\+B$ will consist of infinitely many copies of $T_\beta$ and infinitely many copies of $S_\alpha$.
    
    Note that if $\beta$ is a Harrison order, then $T_\beta \cong T_{\beta+1}$, so $\+A \cong \+B$, and thus our degree will provide an isomorphism mapping the copies of $T_{\beta+1}$ to copies of $T_\beta$. If instead $\beta$ is a well-order, then $T_\beta \not \cong T_{\beta+1}$ (in fact, there is no embedding of $T_{\beta+1}$ into $T_\beta$), so there is no isomorphism from $\+A$ to $\+B$, but $T_{\beta+1}$ $\Sigma_\alpha$-elementarily embeds into $S_\alpha$ (using Lemma 3.3 of \cite{CalvertKnightMillar2006}), so there is a $\Sigma_\alpha$-elementary embedding from $\+A$ to $\+B$.  So our degree will provide such an embedding, mapping the copies of $T_{\beta+1}$ to copies of $S_\alpha$.

    Thus, we can determine whether $\beta$ is a well-order by examining where the provided function sends some root of some copy of $T_{\beta+1}$.
\end{proof}

\section{Relationships Between These Notions}\label{SecRelationships}

We offer some existence results concerning reticence and loquaciousness.

\begin{prop}\label{CollegiateEq}
    There exist degrees of the following kinds:
    \begin{enumerate}
        \item Both reticently high for paths and loquaciously high for paths.
        \item Reticently high for paths but not loquaciously high for paths.
        \item Loquaciously high for paths but not reticently high for paths.
        \item High for paths, but neither reticently high for paths nor loquaciously high for paths.
    \end{enumerate}
\end{prop}

\begin{proof}
 The degrees that are reticently high for paths are exactly those that enumerate $\sO$'s complement (Theorem \ref{thm:reticence_all_the_same}); and a degree that is loquaciously high for paths enumerates $\sO$ (Proposition \ref{prop:lhfp_enumerates_O}). Thus, the degrees which have both of these properties must compute $\sO$. 
 Conversely, $\sO$ can solve the isomorphism problem and so determine whether to give an output or not, so a degree computing $\sO$ is both loquaciously and reticently high for paths.
 
Theorem \ref{LHFInotO} tells us that there are degrees that are loquaciously high for paths and strictly below $\sO$, and Proposition 5.2 of \cite{CFT2023} tells us that there are degrees strictly below $\sO$ that are reticently high for paths. By the previous paragraph, if $\mathbf{d}$ were both loquaciously and reticently high for paths, then $\mathbf{d}$ would compute $\sO$.  Consequently, there exist degrees of each kind that are not of the other.

Finally, there is a degree high for paths and three jumps below $\sO$ as per Proposition 2.13 of \cite{CFT2023}, so it can neither enumerate nor co-enumerate $\sO$ and is thus neither reticently nor loquaciously high for paths.
\end{proof}

In fact, we can strengthen part (4) of this result.

\begin{thm}
    There is a degree that is uniformly high for paths which is neither reticently high for paths nor loquaciously high for paths.
\end{thm}

\begin{proof}
    We fix $(T_e : e \in \omega)$ a standard listing of computable subtrees of $\omega^{<\omega}$. Our forcing conditions are triples $(\sigma, \alpha, \overline{S})$, where:
    \begin{itemize}
        \item $\sigma$ is a finite partial function from $\omega \times \omega$ to $\omega$. We will use the notation $\sigma^{[e]}$ where $\sigma^{[e]}(x) = \sigma(e, x)$. Each $\sigma^{[e]}$ will be defined on a finite initial segment of $\omega$ and will thus be an element of $\omega^{<\omega}$;
        \item If $T_e$ is ill founded, then $\sigma^{[e]}$ is extendable to a path through $T_e$;
        \item $\alpha$ is a computable ordinal;
        \item $\overline{S}$ is a sequence of trees $(S_e : e \in \omega)$ which is low for $\omega_1^{ck}$;
        \item For each $e$, either $S_e = T_e$ or $S_e$ is a nonempty subtree of $T_e$ without leaves; and
        \item $\sigma^{[e]}$ is on $S_e$.
    \end{itemize}

    A condition $(\tau, \beta, \overline{R})$ extends a condition $(\sigma, \alpha, \overline{S})$ if:
    \begin{enumerate}
        \item $\tau$ extends $\sigma$;
        \item $\beta \ge \alpha$;
        \item For each $e$, $R_e$ is a subtree of $S_e$; and
        \item\label{label:uhds_otp_property} For every $e$ with $\tau^{[e]} \neq \sigma^{[e]}$, the rank of $T_e^{\succeq \tau^{[e]}}$ is at least $\alpha$.  (Note that if $T_e$ has a path extending $\tau$, then this rank is $\infty$, which is always greater than $\alpha$.)
    \end{enumerate}
    Here $T_e^{\succeq \tau^{[e]}}$ denotes the tree $\{ \rho \in \omega^{<\omega}: \tau^{[e]}*\rho \in T_e\}$.

    We will take a sufficiently generic filter $\+F$, and then our final object will be the partial function
    \[
    p = \bigcup_{(\sigma, \alpha, \overline{S}) \in \+F} \sigma.
    \]
    More precisely, it will be the graph of this partial function.  Note that if $T_e$ is ill founded, then for any $n$ and any condition $(\sigma, \alpha, \overline{S})$, there is an extension $(\tau, \alpha, \overline{S})$ with $\tau^{[e]}$ having length at least $n$ --- that is, the set of conditions having such extensions is dense.  Thus, if $\+F$ is sufficiently generic, $p^{[e]}$ will be a path through $T_e$.

    As we will show, it will suffice to prove the theorem if $\+O \not \le_e p$ and $\omega \setminus \+O \not \le_e p$, where $\le_e p$ indicates enumeration reducibility to $p$'s graph.\footnote{In general, we identify partial functions with their graphs when considering enumeration reducibility and enumeration operators.}

    In our arguments, we will make use of two variant forms of extension, which we now define. 

    For $(\sigma, \alpha, \overline{S})$ a condition, we say that $\tau$ is a {\em strong extension} if it satisfies:
    \begin{itemize}
        \item $\tau$ extends $\sigma$;
        \item For every $e$, the string $\tau^{[e]}$ is on $S_e$;
        \item For every $e$ with $\tau^{[e]} \neq \sigma^{[e]}$, $\tau^{[e]}$ is extendable to a path through $T_e$.
    \end{itemize}
    So we are strengthening requirement~(\ref{label:uhds_otp_property}) of our definition of extension, and in fact $\alpha$ is irrelevant to this definition. 

    For $(\sigma, \alpha, \overline{S})$ a condition, we say that $\tau$ is a {\em weak extension} if it satisfies:
    \begin{itemize}
        \item $\tau$ extends $\sigma$;
        \item For every $e$, the string $\tau^{[e]}$ is on $S_e$;
        \item For every $e$ with $\tau^{[e]} \neq \sigma^{[e]}$, the rank of $T_e^{\succeq \tau^{[e]}}$ is at least $\alpha$.
    \end{itemize}

    When $(\tau, \alpha, \overline{S})$ is a condition, saying that $\tau$ is a weak extension of $(\sigma, \alpha, \overline{S})$ is equivalent to saying that $(\tau, \alpha, \overline{S})$ is an extension of $(\sigma, \alpha, \overline{S})$; the weakening comes from the fact that $(\tau, \alpha, \overline{S})$ may fail to be a condition because there might be $T_e$ which are ill founded although $\tau^{[e]}$ is not extendable to a path.  Thus, the ``weak" in ``weak extension" has the sense of a ``weak solution" to a differential equation.
    
    We make some observations that will be useful as we proceed.  
    
    \begin{enumerate}
        \item The set $\{ \tau : \text{$\tau$ strongly extends $(\sigma, \alpha, \overline{S})$}\}$ is $\Sigma^1_1(\overline{S})$.
        \item\label{item:weak_extension_sigma_1_1} The set $\{ \tau : \text{$\tau$ weakly extends $(\sigma, \alpha, \overline{S})$}\}$ is $\Sigma^1_1(\overline{S})$ uniformly in $(\sigma, \alpha, \overline{S})$.  All we need to say is that $\tau$ is an extension of $\sigma$, $\tau^{[e]}$ is on $S_e$ for each $e$, and for each $e$ with $\tau^{[e]} \neq \sigma^{[e]}$, $T_e^{\succeq\tau^{[e]}}$ has rank at least $\alpha$.  The last can be said by stating that there is a subtree and a ranking function from the subtree to $\alpha+1$ that gives the root of the subtree rank $\alpha$.
        \item For a fixed computable ordinal $\alpha$, the set $\{ \tau : \text{$\tau$ weakly extends $(\sigma, \alpha, \overline{S})$}\}$ is $\Delta^1_1(\overline{S})$. This is because $\*0^{(2\alpha+3)}$ is able to distinguish the computable trees with rank at least $\alpha$.
        \item Let $\Theta(\tau, \sigma, \alpha, \overline{S})$ be the $\Sigma^1_1$ formula defining weak extension as described in Observation (\ref{item:weak_extension_sigma_1_1}). If $\beta$ is a pseudo-ordinal and $\Theta(\tau, \sigma, \beta, \overline{S})$ holds, then $\tau$ in fact strongly extends $(\sigma, \alpha, \overline{S})$.
        \end{enumerate}

    \smallskip

\begin{claim} $\+O \not \le_e p$.  \end{claim}
    
\begin{proof}    Suppose $(\sigma, \alpha, \overline{S})$ is a condition and $W$ is an enumeration operator. 
    Consider
    \[
        X = \{ n : \exists \tau\, [\text{$\tau$ strongly extends $(\sigma, \alpha, \overline{S})$ and $n \in W^\tau$}]\}.
    \]
    Observe that $X$ is $\Sigma^1_1(\overline{S})$.  $\+O$ is $\Pi^1_1$, and thus $\Pi^1_1(\overline{S})$, but as $\overline{S}$ is low for $\omega_1^{ck}$, we have that $\+O$ cannot be $\Delta^1_1(\overline{S})$.  Thus $\+O \neq X$. So there is some $n \in X\setminus \+O$ or $n \in \+O\setminus X$.

    The first case is that there is some $n \in X \setminus \+O$.  Then fix a witnessing $\tau$, and the condition $(\tau, \alpha, \overline{S})$ ensures that $n \in W^p$ and thus $W^p \neq \+O$.

    The second case is that there is some $n \in \+O \setminus X$.  Fix this $n$.  Consider
    \[
        \{ \beta : \beta \ge \alpha \wedge \exists \tau\, [\text{$\tau$ weakly extends $(\sigma, \beta, \overline{S})$ and $n \in W^\tau$}]\}.
    \]
    This is a $\Sigma^1_1(\overline{S})$ set of ordinals.  It cannot contain any pseudo-ordinals, because then the witnessing $\tau$ would be a strong extension, contradicting $n \not \in X$.  Since $\overline{S}$ is low for $\omega_1^{ck}$, we may apply $\Sigma^1_1$-bounding to get some computable ordinal $\gamma$ such that there is no $\tau$ weakly extending $(\sigma, \gamma, \overline{S})$ with $n \in W^\tau$.  Then the condition $(\sigma, \gamma, \overline{S})$ ensures that $n \not \in W^p$, and thus $W^p \neq \+O$.

    We have shown that there is a dense set of conditions ensuring $W^p \neq \+O$.  It follows that for a sufficiently generic filter, $\+O \not \le_e p$.
\end{proof}

    \begin{claim} $\omega \setminus \+O \not \le_e p$.\end{claim}

\begin{proof}    For $\overline{R}$ a sequence of trees, we will also say $\overline{R}$ {\em extends} $(\sigma, \alpha, \overline{S})$ if $(\sigma, \alpha, \overline{R})$ extends $(\sigma, \alpha, \overline{S})$, though we do not require that $\overline{R}$ is low for $\omega_1^{ck}$.  Note that the class of such extensions is arithmetical in $\overline{S}$.

    Suppose $(\sigma, \alpha, \overline{S})$ is a condition and $W$ is an enumeration operator.
    Consider
    \[
        Y = \{ n : (\forall \overline{R} \text{ extending $(\sigma, \alpha, \overline{S})$})\, \exists \tau\, [\text{$\tau$ weakly extends $(\sigma, \alpha, \overline{R})$ and $n \in W^\tau$}]\}.
    \]
    Note that $\alpha$ is a fixed computable ordinal, so weak extension is $\Delta^1_1$.  Thus $Y$ is $\Pi^1_1(\overline{S})$.  As $\overline{S}$ is low for $\omega_1^{ck}$, we have $Y \neq \omega \setminus \+O$.  So there is some $n \in Y \cap \+O$ or $n \not \in Y \cup \+O$.

    In the first case, fix $n$ and define $\overline{R}$ as follows:
    \begin{itemize}
        \item If $S_e \neq T_e$ or $T_e$ is well-founded, $R_e = S_e$;
        \item Otherwise, let $R_e$ consist of all extendable nodes of $T_e$.
    \end{itemize}
    Note that if $T_e$ is ill founded, then $R_e \subseteq T_e$ contains no leaves, and so every string on $R_e$ is extendable to a path through $T_e$.  Now fix a witnessing $\tau$ for this $\overline{R}$.  Since every string on $R_e$ is extendable to a path through $T_e$, it follows that $(\tau, \alpha, \overline{S})$ is a condition, and so it is a condition extending $(\sigma, \alpha, \overline{S})$.  Further, it ensures $n \in W^p$, and thus $W^P \neq \omega \setminus \+O$.

    In the second case, fix $n$.  We know that
    \[
    \{ \overline{R} : \text{$\overline{R}$ extends $(\sigma, \alpha, \overline{S})$} \wedge \neg \exists \tau\, [\text{$\tau$ weakly extends $(\sigma, \alpha, \overline{R})$ and $n \in W^\tau$}]\}
    \]
    is nonempty, and we note that this class is hyperarithmetical in $\overline{S}$.  By the Gandy Basis Theorem, there is an element $\overline{R}$ which is low for $\omega_1^{ck}$.  So $(\sigma, \alpha, \overline{R})$ is a condition ensuring that $n \not \in W^p$, and thus $W^p \neq \omega \setminus \+O$.

    We have shown that there is a dense set of conditions ensuring $W^p \neq \omega \setminus \+O$.  It follows that for a sufficiently generic filter, $\omega \setminus \+O \not \le_e p$.
\end{proof}

    Thus there is a $p$ such that $\+O \not \le_e p$, $\omega \setminus \+O \not \le_e p$, and for every ill-founded tree $T_e$, the function $p^{[e]}$ gives a path through $T_e$.

    By a strengthening of Selman's Theorem (see e.g.,  Lemma IV.11 of \cite{Montalban1}), there is a Turing degree $\*d$ which enumerates (the graph of) $p$, but $\+O, \omega \setminus \+O \not \in \Sigma^0_1(\*d)$.  From the enumeration of $p$, the degree $\*d$ must be uniformly high for paths.  From Theorem~\ref{thm:reticence_all_the_same}, the degree $\*d$ is not reticently high for paths.  From Proposition~\ref{prop:lhfp_enumerates_O}, the degree $\*d$ is not loquaciously high for paths.  Thus, $\*d$ is our desired degree.
\end{proof}

Thus, we also have that there is a degree that is uniformly high for isomorphism that is neither loquaciously nor reticently high for isomorphism.

\section{Failure modes for non-lowness for isomorphism}\label{SecNonlow}

There is another related notion of structural separation of a degree from $\mathbf{0}$.  We say that a degree is \emph{low for isomorphism} if, whenever it can compute an isomorphism between two computable structures, there is already a computable isomorphism between them \cite{fs-lowim}.  Non-lowness for isomorphism is in this sense not precisely a ``highness" notion, but still identifies a structural class of degrees that are, in a certain sense, separated from $\mathbf{0}$.

We note that the major concepts of this paper trivialize for these degrees.

\begin{prop}  Every degree that is non-low for isomorphism is uniformly, reticently, and loquaciously non-low for isomorphism.
\end{prop}

\begin{proof}
    Let $\mathbf{d}$ be non-low for isomorphism.  We pick one pair of structures $\sA_i \cong \sA_j$ which are not computably isomorphic such that $\mathbf{d}$ computes an isomorphism $f$ between them.  Now, given a pair of indices and an oracle for $\mathbf{d}$, we could do any of the following:
    \begin{enumerate}
        \item If the pair of indices given is $(i,j)$, compute $f$.
        \item If the pair of indices given is $(i,j)$, compute $f$, otherwise do nothing.
        \item If the pair of indices given is $(i,j)$, compute $f$, otherwise compute the constant $0$ function.
    \end{enumerate}
    The first of these procedures demonstrates uniform non-lowness for isomorphism, the second demonstrates reticent non-lowness for isomorphism, and the third demonstrates loquacious non-lowness for isomorphism.
\end{proof}

\section{Answers to other frequently asked questions}\label{sec:other.stuff}

As we mentioned in the introduction, the occasion for this paper was to answer questions we have been asked about \cite{CFT2023}.  In the previous sections, we have answered the questions around ideas we have here formalized as loquacious, reticent, and collegiate senses of highness.  We conclude with answers to some other questions we have been asked about the structural highness notions of the previous paper.

\begin{lem}
    If $\mathbf{d}$ is high for tight descending sequences, then $\mathbf{d}'$ computes $\sO$.
\end{lem}

\begin{proof}
    Let $\sH$ be a Harrison ordering whose well-founded part has degree $\sO$.  Computing a tight descending sequence in $\sH$ with oracle $\mathbf{d}$ allows a $\mathbf{d}$-computable enumeration of the ill-founded part of $\sH$.  So $\mathbf{d}'$ computes the well-founded part of $\sH$ and thus computes $\sO$.
\end{proof}

\begin{prop}
    There is a degree that is high for isomorphism but not high for tight descending sequences.
\end{prop}

\begin{proof}
    We know there is a degree high for isomorphism for which $\sO$ computes the triple jump by Proposition 2.13 in \cite{CFT2023}. Since degrees that are high for descending sequences can be at the most one jump below $\sO$, we are done.
\end{proof}

The following is a restatement of Corollary 4.13 in \cite{CFT2023}, but the corollaries are of note:

\begin{prop}
    There is a degree that is high for tight descending sequences but not high for isomorphism.
\end{prop}

\begin{cor}
    There is a degree that is high for tight descending sequences but not uniformly high for tight descending sequences.
\end{cor}

The following is already known and easily proved by direct diagonalization, but this provides an alternate proof.

\begin{cor}
    The ill-founded part of a computable linear order cannot be $\Sigma^1_1$-complete.
\end{cor}

\bibliographystyle{amsplain}
\bibliography{random,recmodels,universal}

\end{document}